\documentclass[12pt,a4paper]{amsart}

\usepackage{amsmath,amsthm,amsfonts}
\usepackage{hyperref}
\usepackage{graphicx}
\usepackage{tikz}
\usepackage{caption}
\usepackage{subcaption}

\newtheorem{thm}{Theorem}
\newtheorem{lemma}{Lemma}

\newtheorem{defn}{Definition}

\newtheorem*{remark}{Remark}

\begin{document}

\title{Ricci-flat cubic graphs with girth five}

\author[D. Cushing]{David Cushing}
\address{D. Cushing,  Department of Mathematical Sciences, Durham University, Durham DH1 3LE, UK}
\email{david.cushing@durham.ac.uk}

\author[R. Kangaslampi]{Riikka Kangaslampi}
\address{R. Kangaslampi,  Department of Mathematics and Systems Analysis, Aalto University, Aalto  FI-00076, Finland}
\email{riikka.kangaslampi@aalto.fi}

\author[Y. Lin]{Yong Lin}
\address{Y. Lin, Department of Mathematics, School of information, Renmin University of China, Beijing 100872, China}
\email{linyong01@ruc.edu.cn}

\author[S. Liu]{Shiping Liu}
\address{S. Liu, School of Mathematical Sciences, University of Science and Technology of China, Hefei 230026, China}
\email{spliu@ustc.edu.cn}

\author[L. Lu]{Linyuan Lu}
\address{L. Lu, Department of Mathematics, University of South Carolina, Columbia, SC 29208, USA}
\email{lu@math.sc.edu}

\author[S.-T. Yau]{Shing-Tung Yau}
\address{S.-T. Yau, Department of Mathematics, Harvard University, Cambridge, MA 02138, USA}
\email{yau@math.harvard.edu}

\date{\today}

\begin{abstract}
We classify all connected, simple, 3-regular graphs with girth at least 5 that are Ricci-flat. We use the definition of Ricci curvature on graphs given in Lin-Lu-Yau, Tohoku Math., 2011, which is a variation of Ollivier, J. Funct. Anal., 2009. A graph is Ricci-flat, if it has vanishing Ricci curvature on all edges. We show, that the only Ricci-flat cubic graphs with girth at least 5 are the Petersen graph, the Triplex and the dodecahedral graph. This will correct the classification in \cite{LLY14} that misses the Triplex.

\end{abstract}

\maketitle

\section{Introduction}
Ollivier developed a notion of Ricci curvature of Markov chains valid on metric spaces including graphs in \cite{Oll}. The Ollivier-Ricci curvature $\kappa_{p}$ depends on an idleness parameter, $p\in [0,1]$. For graphs, the Ollivier-Ricci curvature has been studied mainly for idleness 0, see \cite{BM, Im, JL, LY10, P}, but for example in \cite{OllVil} Ollivier and Villani considered idleness $\frac{1}{d+1},$ where $d$ is the degree of a regular graph, in order to investigate the curvature of the hypercube. In \cite{LLY11}, Lin, Lu, and Yau modified the definition of Ollivier-Ricci curvature to compute the derivative of the curvature with respect to the idleness, which they denote by $\kappa$.

Throughout this article, let $G=(V,E)$ be a simple graph with vertex set $V$ and edge set $E$. Let $d_x$ denote the degree of the vertex $x\in V$. Let $d(x,y)$ denote the length of the shortest path between two vertices
$x$ and $y$.

We define the following probability distributions $\mu^p_x$ for any
$x\in V,\: p\in[0,1]$:
$$\mu_x^p(z)=\begin{cases}p,&\text{if $z = x$,}\\
\frac{1-p}{d_x},&\text{if $z\sim x$,}\\
0,& \mbox{otherwise.}\end{cases}$$

\begin{defn}
Let $G=(V,E)$ be a locally finite graph. Let $\mu_{1},\mu_{2}$ be two probability measures on $V$. The {\it Wasserstein distance} $W_1(\mu_{1},\mu_{2})$ between $\mu_{1}$ and $\mu_{2}$ is defined as
\begin{equation} \label{eq:W1def}
W_1(\mu_{1},\mu_{2})=\inf_{\pi} \sum_{y\in V}\sum_{x\in V} d(x,y)\pi(x,y),
\end{equation}
where the infimum runs over all transportation plans $\pi:V\times  V\rightarrow [0,1]$ satisfying
$$\mu_{1}(x)=\sum_{y\in V}\pi(x,y),\:\:\:\mu_{2}(y)=\sum_{x\in V}\pi(x,y).$$
\end{defn}

The transportation plan $\pi$ moves a distribution given by $\mu_1$ to a distribution given by
$\mu_2$, and $W_1(\mu_1,\mu_2)$ is a measure for the minimal effort
which is required for such a transition. If $\pi$ attains the infimum in \eqref{eq:W1def} we call it an {\it
  optimal transport plan} transporting $\mu_{1}$ to $\mu_{2}$.

\begin{defn}
The $ p-$Ollivier-Ricci curvature on an edge $x\sim y$ in $G=~(V,E)$ is
$$\kappa_{ p}(x,y)=1-W_1(\mu^{ p}_x,\mu^{ p}_y),$$
where $p$ is called the {\it idleness}.

The Ollivier-Ricci curvature introduced by Lin-Lu-Yau in
\cite{LLY11}, is defined as
$$\kappa(x,y) = \lim_{ p\rightarrow 1}\frac{\kappa_{ p}(x,y)}{1- p}.$$
\end{defn}

We call a graph Ricci-flat if $\kappa(x,y)=0$ on all edges $xy\in G$.  Note that for regular graphs we can calculate the curvature $\kappa(x,y)$ for an edge $xy$ as
$$\kappa(x,y)=\frac{d+1}{d}\kappa_{\frac{1}{d+1}}(x,y),$$
as shown in \cite{I} by Bourne, Cushing, Liu, M\"unch and Peyerimhoff. We show, that the only Ricci-flat cubic graphs with at least girth 5 are the Petersen graph, the Triplex and the dodecahedral graph (see Figure \ref{fig:result}).

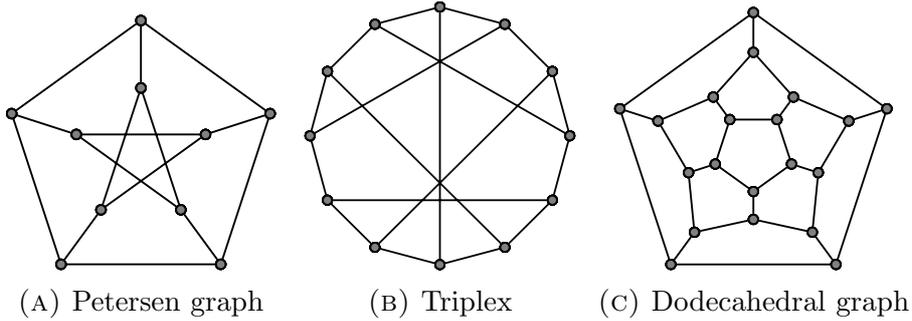
\begin{figure}[h!]
\begin{center}
\tikzstyle{every node}=[circle, draw, fill=black!50,
                        inner sep=0pt, minimum width=4pt]
\begin{subfigure}[b]{0.3\textwidth}
    \centering
        \resizebox{\linewidth}{!}{
 \begin{tikzpicture}[thick,scale=1]%
        \draw \foreach \x in {18,90,...,306} {
        (\x:2)  -- (\x-72:2) 
         
         (90:2) node[label=right:$$](v){}
	 (162:2) node[label=right:$$](u){}
	 (18:2) node[label=left:$$](w){}
	 (234:2) node[label=right:$$](y){}
	 (306:2) node[label=left:$$](x){}
	 (90:1) node[label=right:$$](v1){}
	 (162:1) node[label=right:$$](u1){}
	 (18:1) node[label=left:$$](w1){}
	 (234:1) node[label=right:$$](y1){}
	 (306:1) node[label=left:$$](x1){}
	};	        
\draw (x)--(x1) (y)--(y1) (u)--(u1) (v)--(v1) (w) --(w1);
\draw (x1)--(u1) (y1)--(v1) (u1) --(w1) (v1)--(x1) (w1)--(y1);
\end{tikzpicture} 
}                           
\caption{Petersen graph}
\label{fig:petersen}
    \end{subfigure}
    \begin{subfigure}[b]{0.3\textwidth}
        \centering
        \resizebox{\linewidth}{!}{

 \begin{tikzpicture}[thick,scale=1]%
 \draw \foreach \x in {0,30,...,330} {
        (\x:2)  -- (\x+30:2)    
		(90:2) node(x1){}
		(60:2) node(x2){}
		(30:2) node(x3){}
		(0:2) node(x4){}
		(-30:2) node(x5){}
		(-60:2) node(x6){}
		(120:2) node(x12){}
		(150:2) node(x11){}
		(180:2) node(x10){}
		(210:2) node(x9){}
		(240:2) node(x8){}
		(270:2) node(x7){}		
		};
 \draw (x1)--(x7) (x2)-- (x10) (x3)--(x8) (x4)--(x12) (x5)-- (x9) (x6)-- (x11);		
\end{tikzpicture}    
}                        
\caption{Triplex}
\label{fig:triplex}
    \end{subfigure}
    \begin{subfigure}[b]{0.33\textwidth}
        \centering
        \resizebox{\linewidth}{!}{

 \begin{tikzpicture}[thick,scale=0.6]%
       \draw \foreach \x in {18,90,...,306} {
        (\x+36:1)  -- (\x-36:1)
        (\x:3.5)  -- (\x-72:3.5)
        (270:1) node[label=right:$$](v){}
		(342:1) node[label=right:$$](u){}
		(198:1) node[label=left:$$](w){}
		(54:1) node[label=right:$$](y){}
		(126:1) node[label=left:$$](x){}
	        (270:1.7) node[label=right:$$](v1){}
		(342:1.7) node[label=right:$$](u1){}
		(198:1.7) node[label=left:$$](w1){}
		(54:1.7) node[label=right:$$](y1){}
		(126:1.7) node[label=left:$$](x1){}
	        (306:2.5) node[label=right:$$](v2){}
		(378:2.5) node[label=right:$$](u2){}
		(234:2.5) node[label=left:$$](w2){}
		(90:2.5) node[label=right:$$](y2){}
		(162:2.5) node[label=left:$$](x2){}	
		(306:3.5) node[label=right:$$](v3){}
		(378:3.5) node[label=right:$$](u3){}
		(234:3.5) node[label=left:$$](w3){}
		(90:3.5) node[label=right:$$](y3){}
		(162:3.5) node[label=left:$$](x3){}			
        };
        \draw (v)--(v1) (u)--(u1) (w)--(w1) (x)--(x1) (y)--(y1);
        \draw (y1)--(u2)--(u1)--(v2)--(v1)--(w2)--(w1)--(x2)--(x1)--(y2)--(y1);
        \draw (v2)--(v3) (u2)--(u3) (w2)--(w3) (x2)--(x3) (y2)--(y3);

\end{tikzpicture}    
}                        
\caption{Dodecahedral graph}
\label{fig:dodecahedron}
    \end{subfigure}
   \end{center}
   \caption{The three Ricci-flat cubic graphs with girth 5}
   \label{fig:result}
\end{figure}

\section{Classification of Ricci-flat cubic graphs with girth 5}

In \cite[Theorem 1]{LLY14} the authors classify Ricci-flat graphs with girth at least 5, but there is one 3-regular graph missing from their classification, namely the Triplex  (Figure \ref{fig:triplex}). With the following theorem we can complete the classification. 

\begin{thm}\label{main_thm}
Let $G$ be a  3-regular simple graph with girth $g(G)\geq 5$. If $\kappa(x,y)=0$ on all edges $(x,y)\in G$, then $G$ is the Petersen graph, the Triplex, or the dodecahedral graph.
\end{thm}

Assuming that these graphs can be embedded to surfaces in such a way that they tile the surface with only pentagonal faces, the authors in \cite{LLY14} deduce that they only need to consider graphs with 10 or 20 vertices, obtaining the Petersen graph and the dodecahedral graph. However, the Triplex can be embedded to a torus with three pentagonal, two hexagonal and one 9-gonal face. 
A direct way to fix this problem in the proof of \cite[Theorem 1]{LLY14} is given in \cite{CKLLLY18}.

Using Theorem \ref{main_thm} together with the results for non-cubic graphs from \cite{LLY14} we  have the following classification:

\begin{thm}
Suppose that $G$ is a Ricci-flat graph with girth $g(G) \geq 5$. Then $G$ is one of the following graphs:
 the infinite path,
 cycle $C_n$ with $n \geq 6$,
 the dodecahedral graph,
 the Triplex,
 the Petersen graph and
 the half-dodecahedral graph.
\end{thm} 

Before proving Theorem \ref{main_thm}, let us first consider the local structure of  Ricci-flat cubic graphs. The following lemma shows that in order to have zero curvature on an edge $xy$, the edge must lie on two pentagons.

\begin{lemma}\label{local_structure} Let $G=(V,E)$ be a 3-regular graph of girth $g(G)\ge 5$. If $\kappa(x,y)=0$, then the smallest cycle $C_n$ supporting the edge $xy$ has $n=5$, and in addition $xy$ belongs to two $5$-cycles $P_1$ and $P_2$ such that $P_1\cap P_2=xy$.
\end{lemma}

\begin{proof}
Denote the two other neighbours of $x$ in addition to $y$ by $x_1$ and $x_2$, and the neighbours of $y$ by $y_1$ and $y_2$. We can assume that an optimal transport plan $\pi$ only transports probability distribution from $x_1$ to $y_1$ and $x_2$ to $y_2$. Since there are no triangles or squares, $d(x_i,y_i)\geq 2$. The only possibility to have $\kappa(x,y)=0$ i.e. $\kappa_{1/4}(x,y)= 0$ is if $d(x_i,y_i)=2$, $i=1,2$. Since $d(x_1,y_1)=2$, there exists a path $x_1 u y_1$, where the vertex $u$ cannot be any of the vertices $x_1, x_2,x,y, y_1$ or $y_2$ without forming a triangle or a square. Similarly there exists another vertex $v$ and a path $x_2 y y_2$. Thus we have two pentagons, $x x_1 u y_1 y$ and $x x_2 u y_2 y$, that intersect only on the edge $xy$, as in Figure \ref{fig:start}.

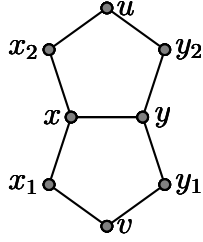
\begin{figure}[h!]
\begin{center}
\tikzstyle{every node}=[circle, draw, fill=black!50,
                        inner sep=0pt, minimum width=4pt]
 \begin{tikzpicture}[thick,scale=0.8]%
 \draw \foreach \x in {18,90,...,306} {
        (\x:1)  -- (\x+72:1)    
		(18:1) node[label=right:$y_2$]{}
		(162:1) node[label=left:$x_2$]{}
		(90:1) node[label=right:$u$]{}
		};
   \begin{scope}[shift={(0,-1.618)}]
       \draw \foreach \x in {18,90,...,306} {
        (\x+36:1)  -- (\x-36:1)
        (270:1) node[label=right:$v$]{}
		(342:1) node[label=right:$y_1$]{}
		(198:1) node[label=left:$x_1$]{}
		(54:1) node[label=right:$y$]{}
		(126:1) node[label=left:$x$]{}
        };
  \end{scope} 
\end{tikzpicture}                            
\end{center}
\caption{Two pentagons that intersect only on $xy$}
\label{fig:startstart}
\end{figure}  
\end{proof}

\begin{remark} Lemma \ref{local_structure} implies that Ricci-flat cubic graphs with girth at least 5 in fact have girth exactly 5.
\end{remark}

Let us now proceed to prove Theorem \ref{main_thm}:

\begin{proof}[Proof of Theorem \ref{main_thm}]
Consider an edge $xy$ in the graph. By lemma \ref{local_structure} we know that this edge is on two pentagons that intersect only on $xy$, as in Figure \ref{fig:startstart}. We will now construct all possible 3-regular, girth 5 simple graphs with $\kappa=0$ on all edges, starting from these pentagons. The vertices $x_2$ and $y_2$ cannot be adjacent to any other vertex in $\{x,y,x_1,y_1,u,v\}$ in order to have girth 5. Thus $x_2$ and $y_2$ are connected to two new vertices, $x_3$ and $y_3$, respectively. The vertex $u$ is either connected to $v$ or to a new vertex $u_1$, as in Figure \ref{fig:start}. Let us consider these two cases separately.

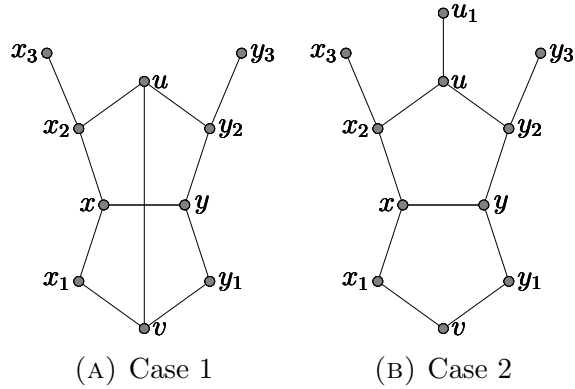
\begin{figure}[h!]
\begin{center}
\tikzstyle{every node}=[circle, draw, fill=black!50,
                        inner sep=0pt, minimum width=4pt]
    \begin{subfigure}[b]{0.3\textwidth}
    \centering
        \resizebox{\linewidth}{!}{
          \begin{tikzpicture}[scale=1]%
 \draw \foreach \x in {18,90,...,306} {
        (\x:1)  -- (\x+72:1)    
		(18:1) node[label=right:$y_2$](y2){}
		(162:1) node[label=left:$x_2$](x2){}
		(90:1) node[label=right:$u$](u){}
		(45:2) node[label=right:$y_3$](y3){}
		(135:2) node[label=left:$x_3$](x3){}
		};
   \begin{scope}[shift={(0,-1.618)}]
       \draw \foreach \x in {18,90,...,306} {
        (\x+36:1)  -- (\x-36:1)
        (270:1) node[label=right:$v$](v){}
		(342:1) node[label=right:$y_1$](y1){}
		(198:1) node[label=left:$x_1$](x1){}
		(54:1) node[label=right:$y$](y){}
		(126:1) node[label=left:$x$](x){}
        }; 
  \end{scope} 
	\draw (u) -- (v);
	\draw (x2) -- (x3);
	\draw (y2) -- (y3);
\end{tikzpicture} 
        }
        \caption{Case 1}   
        \label{fig:case1}
    \end{subfigure}
    \begin{subfigure}[b]{0.3\textwidth}
        \centering
        \resizebox{\linewidth}{!}{
            \begin{tikzpicture}[scale=1]%
 \draw \foreach \x in {18,90,...,306} {
        (\x:1)  -- (\x+72:1)    
		(18:1) node[label=right:$y_2$](y2){}
		(162:1) node[label=left:$x_2$](x2){}
		(90:1) node[label=right:$u$](u){}
		(90:2) node[label=right:$u_1$](u1){}
		(45:2) node[label=right:$y_3$](y3){}
		(135:2) node[label=left:$x_3$](x3){}
		};
   \begin{scope}[shift={(0,-1.618)}]
       \draw \foreach \x in {18,90,...,306} {
        (\x+36:1)  -- (\x-36:1)
        (270:1) node[label=right:$v$]{}
		(342:1) node[label=right:$y_1$]{}
		(198:1) node[label=left:$x_1$]{}
		(54:1) node[label=right:$y$]{}
		(126:1) node[label=left:$x$]{}
        };
  \end{scope} 
	\draw (u) -- (u1);
	\draw (x2) -- (x3);
	\draw (y2) -- (y3);
\end{tikzpicture} 
        }
        \caption{Case 2}
        \label{fig:case2}
    \end{subfigure}
\caption{Two possible graphs to start the construction} 
\label{fig:start}
\end{center}
\end{figure}

\subsubsection*{Case 1:} Since the edge $y_2y_3$ has to lie on two pentagons, two of the following three possible ways for such pentagons to exist must be true (see Figure \ref{fig:case1_cont}): 
\begin{enumerate}
\item[i)] $y_3$ is adjacent to $x_3$, 
\item[ii)] $y_3$ is adjacent to $x_1$ or 
\item[iii)] there is a new vertex $y_4$, adjacent to both $y_3$ and $y_1$.
\end{enumerate} 

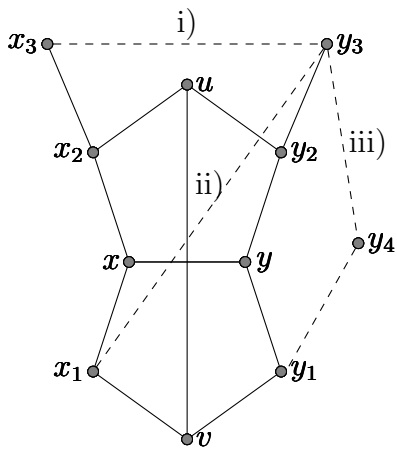
\begin{figure}[h!]
\begin{center}
\tikzstyle{every node}=[circle, draw, fill=black!50,
                        inner sep=0pt, minimum width=4pt]
 \begin{tikzpicture}[scale=1.3]%
 \draw \foreach \x in {18,90,...,306} {
        (\x:1)  -- (\x+72:1)    
		(18:1) node[label=right:$y_2$](y2){}
		(162:1) node[label=left:$x_2$](x2){}
		(90:1) node[label=right:$u$](u){}
		(45:2) node[label=right:$y_3$](y3){}
		(135:2) node[label=left:$x_3$](x3){}
		};
   \begin{scope}[shift={(0,-1.618)}]
       \draw \foreach \x in {18,90,...,306} {
        (\x+36:1)  -- (\x-36:1)
        (270:1) node[label=right:$v$](v){}
		(342:1) node[label=right:$y_1$](y1){}
		(198:1) node[label=left:$x_1$](x1){}
		(54:1) node[label=right:$y$](y){}
		(126:1) node[label=left:$x$](x){}
		(30:2) node[label=right:$y_4$](y4){}
        }; 
  \end{scope} 
	\draw (u) -- (v);
	\draw (x2) -- (x3);
	\draw (y2) -- (y3);
	\draw [dashed] (x3) -- node[draw=none,fill=none,above] {i)} ++(y3) ;
	\draw [dashed] (x1) -- node[draw=none,fill=none,above] {ii)} ++(y3);
	\draw [dashed] (y3) -- node[draw=none,fill=none,right] {iii)} ++(y4) -- (y1);
\end{tikzpicture}                      
           
\end{center}
\caption{The three ways to construct a pentagon with $y_2y_3$ in Case 1}
\label{fig:case1_cont}
\end{figure}

The cases ii) and iii) cannot be true at the same time, since then it would not be possible to continue the construction in order to have pentagons with the edge $(x_2,x_3)$: the only vertex with degree less than three, $y_4$, would be too far from $x_3$. 

If i) and ii) are true, then the only vertices with degree less that 3 are $x_3$ and $y_1$, and $d(x_3,y_1)=4$. Thus, the last edge from $x_3$ must go to $y_1$ for it to be on a pentagon. Now all edges lie on two pentagons, and we have constructed the Petersen graph with 10 vertices. 

If i) and iii) are true, then there must be another vertex $x_4$ which is adjacent to $x_1$, since $x_1$ cannot be adjacent to either of the existing vertices with degree less than 3, that is, $x_3$ or $y_4$, without forming a square. Let us then construct pentagons with $x_1x_4$. There are two possibilities for the first pentagon: either $x_4$ is adjacent to $x_3$ or to $y_4$. Consider the first possibility. Then $x_4$ cannot be adjacent to $y_4$, since that would create a square $x_4x_3y_3y_4$. So there must be yet another vertex $x_5$ to which $x_4$ is adjacent, and which is adjacent to $y_4$. But now all other vertices in the graph but $x_5$ have degree 3, and the construction cannot be continued. Consider then the second possibility, $x_4$ being adjacent to $y_4$. Similarly, now $x_4$ cannot be adjacent to $x_3$, we need a new vertex $x_5\sim x_4$. Then the only other vertex with degree less than 3 is $x_3$, and $d(x_5,x_3)=4$, so $x_5$ must be adjacent to $x_3$. But that leaves $x_5$ the only vertex with degree less than 3, and the construction cannot be continued. Thus, i) and iii) cannot be true at the same time, and the only graph with the edge $uv$ is the Petersen graph.

\subsubsection*{Case 2:} Since the edge $u u_1$ must lie on two pentagons, one of them through $x_2$ and another through $y_2$, we have two isomorphically different possibilities: \begin{enumerate}
\item[a)]there is one new vertex $x_4$ adjacent to $u_1$ and $x_3$ and $u_1$ is also adjacent to $y_1$ (Figure \ref{fig:case2a}) or
\item[b)] there are two new vertices, $x_4$ adjacent to $u_1$ and $x_3$ and $y_4$ adjacent to $u_1$ and $y_3$ (Figure \ref{fig:case2b}).
\end{enumerate}

If $u_1$ were adjacent to both $x_1$ and $y_1$, that would create a square $u_1 y_1 v x_1$. The case where $u_1$ would be adjacent to $x_1$ and to a vertex $y_4$ with $y_4\sim y_3$ is isomorphic to the case a) above.
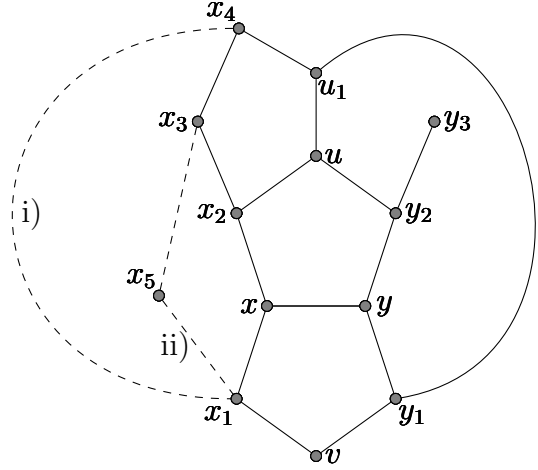
\begin{figure}[h!]
\begin{center}
\tikzstyle{every node}=[circle, draw, fill=black!50,
                        inner sep=0pt, minimum width=4pt]
 \begin{tikzpicture}[scale=1.1]%
 \draw \foreach \x in {18,90,...,306} {
        (\x:1)  -- (\x+72:1)    
		(18:1) node[label=right:$y_2$](y2){}
		(162:1) node[label=left:$x_2$](x2){}
		(90:1) node[label=right:$u$](u){}
		(90:2) node[label=below right:$u_1$](u1){}
		(45:2) node[label=right:$y_3$](y3){}
		(135:2) node[label=left:$x_3$](x3){}
		(110:2.7) node[label=above left:$x_4$](x4){}
		(200:2) node[label=above left:$x_5$](x5){}
		};
   \begin{scope}[shift={(0,-1.618)}]
       \draw \foreach \x in {18,90,...,306} {
        (\x+36:1)  -- (\x-36:1)
        (270:1) node[label=right:$v$](v){}
		(342:1) node[label=below right:$y_1$](y1){}
		(198:1) node[label=below left:$x_1$](x1){}
		(54:1) node[label=right:$y$](y){}
		(126:1) node[label=left:$x$](x){}
	  }; 
  \end{scope} 
	\draw (u) -- (u1); \draw (u1) -- (x4);  
	\draw (x2) -- (x3); \draw (x4) -- (x3);
	\draw (y2) -- (y3); 
	\draw [dashed] (x3)-- (x5) -- node[draw=none,fill=none,left] {ii)} ++ (x1);
	\draw (u1) to [out=40,in=10, distance=3cm] (y1);
	\draw [dashed] (x4) to [out=180,in=180, distance=3.5cm] node[draw=none,fill=none,right] {i)} (x1) ;
\end{tikzpicture}                      
        \caption{Case 2 a) with two non-isomorphic ways to continue}
        \label{fig:case2a}
        \end{center}
    \end{figure}
    Assume that a) is true. Then there are two isomorphically different possibilities to have a pentagon through $x_3 x_2 x$, illustrated in Figure \ref{fig:case2a}: 
\begin{enumerate}
\item[i)] Assume $x_4\sim x_1$. Then in order to have two pentagons with $y_2y_3$, we must have $y_3\sim x_3$ and $y_3\sim v$, which gives us the Triplex.
\item[ii)] Assume that there is a new vertex $x_5$ adjacent to $x_1$ and $x_3$. Then there can exist two pentagons with $y_2y_3$ only if $y_3\sim x_4$ and $y_3\sim v$. But then the graph cannot be continued from $x_5$, since all other vertices already have degree 3.
\end{enumerate}
Note that having $x_3\sim v$ would also create a pentagon with $x_3 x_2 x$, but that is in fact isomorphic to i). 

 \begin{figure}[h!]
\begin{center}
\tikzstyle{every node}=[circle, draw, fill=black!50,
                        inner sep=0pt, minimum width=4pt]
 \begin{tikzpicture}[scale=1.1]%
       \draw \foreach \x in {18,90,...,306} {
        (\x:1)  -- (\x+72:1)    
		(18:1) node[label=right:$y_2$](y2){}
		(162:1) node[label=left:$x_2$](x2){}
		(90:1) node[label=right:$u$](u){}
		(90:2) node[label=below right:$u_1$](u1){}
		(45:2) node[label=right:$y_3$](y3){}
		(135:2) node[label=left:$x_3$](x3){}
		(110:2.7) node[label=above left:$x_4$](x4){}
		(70:2.7) node[label=above right:$y_4$](y4){}
		(200:2) node[label=above left:$x_5$](x5){}
		(-20:2) node[label=above right:$y_5$](y5){}
		};
   \begin{scope}[shift={(0,-1.618)}]
       \draw \foreach \x in {18,90,...,306} {
        (\x+36:1)  -- (\x-36:1)
        (270:1) node[label=below:$v$](v){}
		(342:1) node[label=below right:$y_1$](y1){}
		(198:1) node[label=below left:$x_1$](x1){}
		(54:1) node[label=right:$y$](y){}
		(126:1) node[label=left:$x$](x){}
	  }; 
  \end{scope} 
	\draw (u) -- (u1);\draw (u1) -- (x4); \draw (u1) -- (y4); 
	\draw (x2) -- (x3);\draw (x4) -- (x3);
	\draw (y2) -- (y3);\draw (y4) -- (y3);
	\draw [dashed] (x3)-- (x5) -- node[draw=none,fill=none,above] {iii)} ++ (x1);
		\draw [dashed] (x4) to [out=180,in=180, distance=3.5cm] node[draw=none,fill=none,right] {i)} (x1) ;
	\draw [dashed] (x3) to [out=200,in=200, distance=2.5cm] node[draw=none,fill=none,right] {ii)} (v) ;
	\draw [dashed] (y3)-- (y5) -- node[draw=none,fill=none,above] {vi)} ++ (y1);
		\draw [dashed] (y4) to [out=0,in=0, distance=3.5cm] node[draw=none,fill=none,right] {iv)} (y1) ;
	\draw [dashed] (y3) to [out=-20,in=-20, distance=2.5cm] node[draw=none,fill=none,left] {v)} (v) ;
\end{tikzpicture}                
        \caption{Case 2 b) with possible continuations}
        \label{fig:case2b}
        \end{center}
    \end{figure}
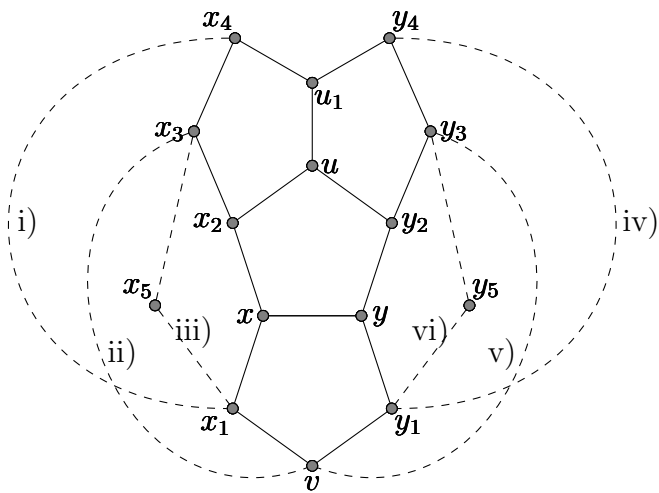
Assume then that b) is true. Then there are three possibilities to have a pentagon through $x_2 x x_1$, illustrated in Figure \ref{fig:case2b} with labels i), ii) and iii). Symmetrically, there is three possibilities to have a pentagon with $y_2 y y_1$, illustrated in Figure \ref{fig:case2b} with labels iv), v) and vi). Let us consider the possible non-isomorphic cases with pentagons on $x_2 x x_1$ and $y_2 y y_1$. There are four of them, since i) \& v) is just a mirror image of ii) \& iv), the four combinations i) \& vi), its mirror iii) \& iv), ii) \& vi) and its mirror iii) \& v) all give in fact the same graph, and the case ii) \& v) would require $d(v)=4$.
\begin{enumerate}
\item[i) \& iv)] Assume that $x_4\sim x_1$ and $y_4 \sim y_1$. Then in order to have a pentagon with $x_4 u_1 y_4$ the only possibility is that $x_3\sim y_3$. But that leaves only the vertex $v$ with degree less than three, and the construction cannot be continued.  
\item[i) \& v)] Assume that $x_4\sim x_1$ and $y_3\sim v$. Then to have a pentagon with $x_4 x_1 v$ we must have $x_3 \sim y_1$. But that leaves only the vertex $y_4$ with degree less than three, and the construction cannot be continued.  
\item[i) \& vi)] Assume that $x_4\sim x_1$ and that there exists a vertex $y_5$ such that $y_3\sim y_5 \sim y_1$. Then to have a pentagon with $u_1 x_4 x_1$ we must have $y_4 \sim v$. But that leaves only the vertices $x_3$ and $y_5$ with degree less than three, and since now $d(x_3,y_5)=5$, the construction cannot be continued.
\item[iii) \& vi)] Assume that there exists a vertex $x_5$ such that $x_3\sim x_5 \sim x_1$ and that a vertex $y_5$ such that $y_3\sim y_5 \sim y_1$. Then there is two non-isomorphic ways to have a pentagon with $v y_1 y_5$: either $x_5\sim y_5$ or there exists two new verties $v_1$ and $y_6$ such that $v\sim v_1\sim y_6\sim y_5$. In the first case we must then have $v\sim y_4$ to have a pentagon with $y_4 y_3 y_5$. But then the only vertex with degree less than three is $x_4$, and the construction cannot be continued. In the latter case (Figure \ref{fig:toDode}) there is two possibilities to have a pentagon with $x_5 x_1 v$: either $x_5\sim y_6$, or there is a new vertex $x_6$ with $x_5\sim x_6 \sim v_1$.

 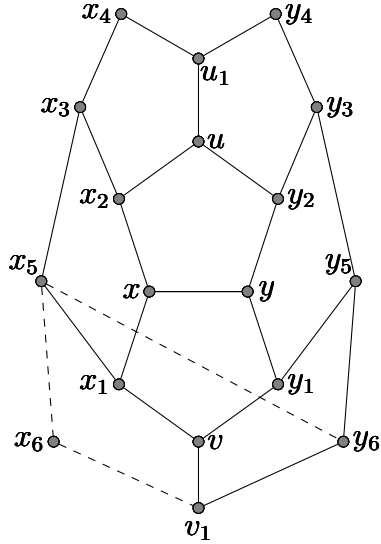
\begin{figure}[h!]
\begin{center}
\tikzstyle{every node}=[circle, draw, fill=black!50,
                        inner sep=0pt, minimum width=4pt]
 \begin{tikzpicture}[scale=1.1]%
       \draw \foreach \x in {18,90,...,306} {
        (\x:1)  -- (\x+72:1)    
		(18:1) node[label=right:$y_2$](y2){}
		(162:1) node[label=left:$x_2$](x2){}
		(90:1) node[label=right:$u$](u){}
		(90:2) node[label=below right:$u_1$](u1){}
		(45:2) node[label=right:$y_3$](y3){}
		(135:2) node[label=left:$x_3$](x3){}
		(110:2.7) node[label=left:$x_4$](x4){}
		(70:2.7) node[label=right:$y_4$](y4){}
		(200:2) node[label=above left:$x_5$](x5){}
		(-20:2) node[label=above left:$y_5$](y5){}
		};
   \begin{scope}[shift={(0,-1.618)}]
       \draw \foreach \x in {18,90,...,306} {
        (\x+36:1)  -- (\x-36:1)
        (270:1) node[label=right:$v$](v){}
		(342:1) node[label=right:$y_1$](y1){}
		(198:1) node[label=left:$x_1$](x1){}
		(54:1) node[label=right:$y$](y){}
		(126:1) node[label=left:$x$](x){}
		(270:1.8) node[label=below:$v_1$](v1){}
		(330:2) node[label=right:$y_6$](y6){}
		(210:2) node[label=left:$x_6$](x6){}
	  }; 
  \end{scope} 
	\draw (u) -- (u1);\draw (u1) -- (x4); \draw (u1) -- (y4); 
	\draw (x2) -- (x3);\draw (x4) -- (x3);
	\draw (y2) -- (y3);\draw (y4) -- (y3);
	\draw  (x3)-- (x5) --  (x1);
	\draw  (y3)-- (y5) -- (y1);
	\draw (v) -- (v1) -- (y6) -- (y5);
	\draw [dashed] (x5) -- (x6) -- (v1) ;
	\draw [dashed] (x5) -- (y6);
\end{tikzpicture}                
        \caption{The graph for case 2 b) with iii), vi) and vertices $v_1$ and $y_6$ have two possible ways to continue}
        \label{fig:toDode}
        \end{center}
    \end{figure}

Assume first that $x_5\sim y_6$. Then there are three vertices with degree less than three,  $v_1$, $x_4$ and $y_4$. Since $d(v_1,x_4)=d(v_1,y_4)=4$, $v_1$ must be adjacent to either $x_4$ or $y_4$. But that leaves only one vertex with degree less than three, and the construction cannot be continued. Thus we must have $x_5\sim x_6 \sim v_1$. To have a pentagon with $x_4 x_3 x_5$ we need yet another vertex $x_7$ with $x_4\sim x_7\sim x_6$. Similarly for a pentagon with $y_4 y_3 y_5$ we need a vertex $y_7$ with $y_4\sim y_7\sim y_6$. Then there are two vertices with degree less than three in the graph, $x_7$ and $y_7$, and connecting them with an edge finalizes the construction by creating the dodecahedral graph. 
\end{enumerate}
We have now shown that the only 3-regular graphs with girth 5 and with every edge on two pentagons intersecting only at that edge are  the Petersen graph, the Triplex and the dodecahedral graph. Therefore these are the only 3-regular Ricci-flat graphs with girth at least 5.
\end{proof}

We also calculated the curvatures of all cubic graphs with 20 vertices or less, and obtained the same three Ricci-flat graphs with girth 5. We  used the graph generator package \texttt{nauty} by B. D. McKay and A. Piperno \cite{MP} to generate the graphs and the Graph Curvature Calculator \cite{calculator} to calculate the curvatures.

\end{document}